\documentclass{amsart} 
\usepackage{graphicx}
\usepackage{color}
\usepackage{amsmath}
\usepackage{amsthm}
\usepackage{amssymb}
\usepackage{amsfonts}

\newtheorem{theorem}{Theorem}[section]
\newtheorem{proposition}[theorem]{Proposition}

\newtheorem{lemma}[theorem]{Lemma}

\newcommand{\affil}[2]{}
\newcommand{\barticle}{}
\newcommand{\earticle}{}

\usepackage{tikz}
\usetikzlibrary{decorations.markings}

\DeclareMathOperator{\Lk}{Lk}
\DeclareMathOperator{\supp}{supp}
\DeclareMathOperator{\dist}{Dist}
\DeclareMathOperator{\FV}{FV}

\DeclareMathOperator{\FA}{FA}
\newcommand{\Z}{\mathbb{Z}}
\newcommand{\N}{\mathbb{N}}
\newcommand{\R}{\mathbb{R}}
\DeclareMathOperator{\alk}{\Lk_\uparrow}
\DeclareMathOperator{\dlk}{\Lk_\downarrow}
\DeclareMathOperator{\area}{Area}
\DeclareMathOperator{\vol}{Vol}
\DeclareMathOperator{\mass}{Mass}

\newcommand{\D}{\Delta}

\theoremstyle{definition}
\newtheorem*{remark}{Remark}

\begin{document}

\title{Homological and homotopical Dehn functions are different}
\author[A. Abrams]{Aaron Abrams}
\address{
Aaron Abrams\\
Mathematics Department\\
Robinson Hall\\ 
Washington and Lee University\\
Lexington VA 24450}
\email{abrams.aaron@gmail.com}
\author[N. Brady]{Noel Brady}
\address{
Noel Brady\\
Department of Mathematics\\
University of Oklahoma\\
601 Elm Ave\\
Norman, OK 73019}
\email{nbrady@math.ou.edu}
\author[P. Dani]{Pallavi Dani}
\address{
Pallavi Dani\\
Department of Mathematics\\
Louisiana State University\\
Baton Rouge, LA 70803-4918}
\email{pdani@math.lsu.edu}
\author[R. Young]{Robert Young}
\address{
Robert Young\\
Department of Mathematics\\
University of Toronto\\
40 St.\ George St., Room 6290\\
Toronto, ON  M5S 2E4\\
Canada}
\email{ryoung@math.toronto.edu}

% \author{A. Abrams\affil{1}{Washington and Lee University, Lexington VA, USA},
% N. Brady\affil{2}{University of Oklahoma, Norman, OK, USA},
% P. Dani\affil{3}{Louisiana State University, Baton Rouge, LA, USA},
% R. Young\affil{4}{University of Toronto, Toronto, Canada}}

\thanks{
The authors are thankful to the American Institute of Mathematics for supporting
this research and to Dan Guralnik and Sang Rae Lee for helpful discussions. We gratefully acknowledge additional funding from National Science Foundation award DMS-0906962 (to NB), the Louisiana Board of Regents %, through the Board of Regents 
Support Fund contract LEQSF(2011-14)-RD-A-06 (to PD), and the
Natural Sciences and Engineering Research Council of Canada (to RY)}

\barticle

\begin{abstract}
  The homological and homotopical Dehn functions are different
  ways of measuring the difficulty of filling a closed curve inside a
  group or a space.  The homological Dehn function measures fillings of cycles by
  chains, while the homotopical Dehn function measures fillings of
  curves by disks.  Since the two definitions involve different sorts
  of boundaries and fillings, there is no \emph{a priori} relationship
  between the two functions, but prior to this work there were no known examples of
  finitely-presented groups for which the two functions differ.  This paper gives the
  first such examples, constructed by amalgamating a free-by-cyclic
  group with several Bestvina-Brady groups.
\end{abstract}

 \maketitle

\section{Introduction}

The classical isoperimetric problem is to determine the
maximum area that can be enclosed by a closed curve of fixed length in the plane.
This problem has been generalized in many different ways. 
For example, in a metric space $X$, one can study the homotopical filling area
of a curve $\gamma$, denoted $\delta_X(\gamma)$ and defined to be the infimal area of a disk
whose boundary is $\gamma$.  This leads to the idea of the \emph{homotopical Dehn function}
of $X$, defined as the
smallest function $\delta_X$ such that any closed curve $\gamma$
of length $\ell$ has filling area at most $\delta_X(\ell)$.  A
remarkable result of Gromov \cite{GroAsymp, BridsonSurv} states that if $X$ is simply connected and
if there is a group $G$ that acts on $X$ geometrically (i.e.,
cocompactly, properly discontinuously, and by isometries), then the
growth rate of $\delta_X$ depends only on $G$; indeed, $\delta_X$ is
related to the difficulty of determining whether a product of
generators of $G$ represents the identity.  We can thus define the
Dehn function of a group to be the (homotopical) Dehn function of any
simply-connected space that the group acts on geometrically; this is
well-defined up to some constants (for details, see
Sec.~\ref{sec:prelims}).

Another way to generalize the isoperimetric problem is to consider
fillings of 1-cycles by 2-chains instead of fillings of curves by disks.
If $\alpha$ is a $1$-cycle in $X$, we can define its homological
filling area $\FA(\alpha)$ to be the infimal mass of a $2$-chain in
$X$ with integer coefficients whose boundary is $\alpha$.  This leads to
the \emph{homological Dehn function} $\FA_X$, defined as the smallest 
function such that any $1$-chain $\alpha$ of mass at most $\ell$
has a homological filling of area at most $\FA_X(\ell)$.  Like its
homotopical counterpart, $\FA_X$ can be used to construct a group
invariant:  if $H_1(X)=0$ and if there is a group $G$ that acts on $X$
geometrically, then the growth rate of $\FA_X$ depends only on $G$ and
we can define $\FA_G = \FA_X$.  Again this is well-defined up to 
constants.

The exact relationship between these two filling functions has
been an open question for some time.  The homological Dehn function
deals with a wider class of possible fillings (surfaces of arbitrary
genus) and a wider class of possible boundaries (sums of arbitrarily
many disjoint closed curves), so it is not \emph{a priori} clear
whether $\FA_H$ is always the same as $\delta_H$ when they are both
defined.  Some hints that they may differ come from a construction
of groups with unusual finiteness properties due to Bestvina and Brady
\cite{BestBrad}.  They used a combinatorial version of Morse theory to
construct a group which is $FP_2$ but not finitely presented.  Such a
group does not act geometrically on any simply connected space but
does acts geometrically on a space with trivial first homology, so its
homological Dehn function is defined, but its homotopical Dehn
function is undefined.

In this paper, we will construct a family of finitely presented groups
such that $\FA_H$ grows strictly slower than $\delta_H$.
Specifically, we will show:
\begin{theorem}\label{thm:THM}
  For every $d\in\N\cup\{\infty\}$, there is a CAT(0) group $G$
  containing a finitely presented subgroup $H$ such that 
  $\FA_H(\ell)\preceq \ell^5$ and the homotopical Dehn function 
  satisfies
  \begin{align*}
   & \ell^d \preceq \delta_H(\ell)%\preceq \ell^{d+3} 
   &&  \text{if }d\in\N,\\
   & e^\ell \preceq \delta_H(\ell) && \text{if }d=\infty.
  \end{align*}
\end{theorem}

\begin{remark}  Using methods of \cite{BrGuLe}, one can show that in 
the $d\in\N$ case, the group $H$
constructed in the theorem satisfies $\delta_H(\ell)\preceq \ell^{d+3}$.
\end{remark}

Our construction uses methods of Brady, Guralnik, and
Lee \cite{BrGuLe} to create a hybrid of a Bestvina-Brady group
with a group having large Dehn function.  The resulting group is finitely presented, so 
both $\delta_H$ and $\FA_H$ are defined, and we
will show that the unusual finiteness properties coming from the
Bestvina-Brady construction lead to a large gap between homological
and homotopical filling functions.  

Similar results are known for higher-dimensional versions of $\delta$
and $\FA$.  One can define $k$-dimensional homotopical and homological
Dehn functions by considering fillings of $k$-spheres or
$k$-cycles by $(k+1)$-balls or $(k+1)$-chains; by historical accident,
the corresponding homotopical and homological filling functions
have come to be called $\delta^k_X$ and $\FV^{k+1}_X$ respectively.  
The relationship
between $\delta^k_X$ and $\FV^{k+1}_X$ is better understood when
$k\ge 2$, because in this case the Hurewicz theorem can be used to 
replace cycles and chains by spheres and balls.

If $X$ is $k$-connected and $\beta$ is a $(k+1)$-chain with $k\ge 2$,
then the Hurewicz theorem can be used to show that $\beta$ is the
image of the fundamental class of a ball under a map $b:D^{k+1}\to X$
with $\vol b=\mass \beta$.  Thus, if $a:S^{k}\to X$ is a map of a
sphere and $\alpha$ is the image of the fundamental class of $S^{k}$
under $a$, then $\delta^{k}_X(a)= \FV^{k+1}_X(\alpha)$, so
$\delta^{k}_X\preceq\FV^{k+1}_X$ for $k\ge 2$ (see Appendix 2 of
\cite{GroFRM} and \cite{White}).

Likewise, if $X$ is $k$-connected and $\alpha$ is a $k$-cycle for
$k\ge 3$, then the Hurewicz theorem can be used to show that $\alpha$
is the image of the fundamental class of a sphere under a map
$a:S^k\to X$ such that $\vol a=\mass \alpha$ (see Remark 2.6.(4) of
\cite{BraBriForSha}).  Consequently, since
$\delta^{k}_X(a)=\FV^{k+1}_X(\alpha)$, we have $\FV^{k+1}_X\sim
\delta^{k}_X$ for $k\ge 3$.

Thus, if $k\ge 3$ and if $H$ is a group which acts geometrically on a
$k$-connected complex, the above results imply that
$\delta^{k}_H\sim\FV^{k+1}_H$.  When $k=2$, Young constructed examples
of groups for which $\delta^{2}_H\precnsim\FV^{3}_H$ \cite{YoungHom}.
The examples in this paper are the first known examples of groups for
which $\delta^1_H\not\sim\FV^{2}_H$.

% what's left?  acknowledgements?

% Add a sentence about notation

% reference Papasoglu

% figure for upper bounds?

\section{Preliminaries}\label{sec:prelims}

\subsection{Dehn functions}
For a full exposition of Dehn functions, consult \cite{BridsonSurv}.
We will briefly review the definitions that we will need.  Let $X$ be
a simply connected riemannian manifold or simplicial complex.  If
$\alpha:S^1\to X$ is a Lipschitz map, define the \emph{homotopical
  filling area} of $\alpha$ to be
$$\delta_X(\alpha)=\mathop{\inf_{\beta:D^2\to
    X}}_{\beta|_{S^1}=\alpha}\area \beta,$$ where $\beta$ ranges over
Lipschitz maps $D^2\to X$ which agree with $\alpha$ on $\partial
D^2$.  Since $X$ is simply-connected and any continuous map can be
approximated by a Lipschitz map, such maps exist.
We can define an invariant of $X$ by letting
$$\delta_X(\ell)=\sup_{\ell(\alpha)\le \ell}\delta_X(\alpha).$$
We call this the \emph{homotopical Dehn function} of $X$.

We define a relation on functions $\R^{\ge 0}\to \R^{\ge 0}$ 
by $f\preceq g$ if there is a $c>0$ such that for all $n$,
$$f(n)\le c g(c n+c)+cn+c.$$
If $f\preceq g$ and $g\preceq f$, we write $f\sim g$.  Thus
$\sim$ distinguishes all functions $x^\mu$ for $\mu\ge 1$,
and all functions of the form $\lambda^x$ are equivalent for $\lambda>1$.
%Roughly, if $f\sim g$, then $f$ and $g$ grow at the same rate. 
Gromov stated (and Bridson proved in \cite{BridsonSurv}) that
if $H$ acts geometrically on $X$ (for instance, if $H=\pi_1(M)$ and
$X=\widetilde{M}$ for some compact $M$), then $\delta_X(n)$ is determined
up to $\sim$-equivalence by $H$.  If $H$ is finitely presented, then $H$ acts
geometrically on the universal cover $\widetilde{X}_H$ of a presentation
complex, which is a 2-complex $X_H$ with $\pi_1(X_H)=H$.  Thus
$\delta_H:=\delta_{\widetilde{X}_H}$ is well-defined up to $\sim$.

To define the homological invariant $FA_H$, suppose that $X$ is
a polyhedral complex with $H_1(X)=0$.  If
$\alpha$ is a 1-cycle in $X$, we let
$$\FA_X(\alpha)=\inf_{\partial \beta=\alpha}\mass \beta,$$
where $\mass \beta$ is defined to be $\sum |b_i|$ if $\beta=\sum b_i \Delta_i$ 
is a sum of faces of $X$ with integer coefficients.
We can define an invariant of $X$ by letting
$$\FA_X(\ell)=\sup_{\mass \alpha\le \ell}\FA_X(\alpha).$$
We call this the \emph{homological Dehn function} of $X$.  Like the
homotopical Dehn function, if $H$ acts geometrically on $X$, then
$\FA_X$ is determined up to $\sim$ by $H$, and if $X_H$
is a presentation complex for a finitely presented group $H$, we define
$\FA_H=\FA_{\widetilde{X}_H}$.

\subsection{Right-angled Artin groups}\label{subsec:RAAG}
If $\Lambda$ is a simple graph (i.e., without loops or multiple
edges), we can define a \emph{right-angled Artin group} (RAAG) based on
$\Lambda$.  If $V(\Lambda)$ and $E(\Lambda)$ are the vertex set and
edge set of $\Lambda$, we define 
$$A_\Lambda=\langle V(\Lambda)\mid [i(e),t(e)]=1 \text{ for all $e\in
  E(\Lambda)$}\rangle,$$
where $i(e)$ and $t(e)$ are the endpoints of $e$.  We say that $\Lambda$ is the
\emph{defining graph} of $A_\Lambda$.  These RAAGs generalize free
groups and free abelian groups; if $\Lambda$ is a complete graph, there is
an edge between every pair of vertices, so every pair of generators of
$A_\Lambda$ commutes and $A_\Lambda$ is free abelian.  On the other hand, if
$\Lambda$ has no edges, then $A_\Lambda$ is a free group.

A full exposition of RAAGs can be found in \cite{Charney}.  One
important fact that we will use is that for every $\Lambda$, there is
a one-vertex locally CAT(0) cube complex $X_\Lambda$ with
$\pi_1(X_\Lambda)=A_\Lambda$; this is called the \emph{Salvetti
  complex}.  This complex can be built directly from the graph
$\Lambda$; it has one vertex, one edge for every vertex of $\Lambda$,
one square for each edge of $\Lambda$, and one $n$-cube for each
$n$-vertex clique in $\Lambda$.

Bestvina and Brady \cite{BestBrad} used RAAGs to construct subgroups
of nonpositively-curved groups with unusual finiteness properties.
They defined a homomorphism $h_{A_\Lambda}:A_\Lambda\to \Z$ which sends each
generator of $A_\Lambda$ to 1 and, viewing $A_\Lambda$ as the 
$0$-skeleton of $\widetilde{X}_\Lambda$, extended it to a map
$h_{X_\Lambda}:\widetilde{X}_\Lambda\to \R$.  This map is linear on each cube of
$\widetilde{X}_\Lambda$, so the level set $L_{A_\Lambda}:=h_{X_\Lambda}^{-1}(0)$
can be given the structure of a polyhedral complex.  The subgroup
$H_{A_\Lambda}:=\ker h_{A_\Lambda}$ acts vertex-transitively on $L_{A_\Lambda}$, so if
$L_{A_\Lambda}$ is connected, the $1$-skeleton of $L_{A_\Lambda}$ is a
Cayley graph for $H_{A_\Lambda}$.  In this case, we can construct a generating set
for $H_{A_\Lambda}$ explicitly: each edge of $L_{A_\Lambda}$ is a diagonal of a
square of $\widetilde{X}_\Lambda$, so $H_{A_\Lambda}$ has a generating set
consisting of elements of the form $ab^{-1}$, where $a$ and $b$ are
generators of $A_\Lambda$.

Recall that a complex is \emph{flag} if every clique of $n$ vertices spans an $(n-1)$-dimensional simplex.
Bestvina and Brady proved that the topology of $\Lambda$ determines
the topology of $h_{X_\Lambda}^{-1}(0)$:
\begin{theorem}[\cite{BestBrad}]\label{thm:BestBrad}
  If $\Lambda$ is the 1-skeleton of a flag complex $Y$, and 
  $h_{A_\Lambda},h_{X_\Lambda}$ are the
  maps defined above, then $H_{A_\Lambda}=\ker h_{A_\Lambda}$ 
  acts on the complex $L_{A_\Lambda}=h_{X_\Lambda}^{-1}(0)$, which 
  is homotopy equivalent to a 
  wedge product of infinitely many copies of $Y$, indexed by the vertices in 
  $\widetilde{X}_\Lambda\setminus L_{A_\Lambda}$.  
  In fact, $h_{X_\Lambda}^{-1}(0)$ is a union of infinitely many scaled copies of $Y$.
\end{theorem}

The main tool used to prove this theorem is a combinatorial version of
Morse theory.  If $X$ is a complex, $x\in X$ is a vertex, and $h:X\to
\R$ is a function which is linear on each cell and is not constant on
any edge, one may define subcomplexes $\alk(x)$ and $\dlk(x)$ of the link
$\Lk(x)$ called the \emph{ascending} and \emph{descending links} of $x$.  To define these, we
identify the vertices of $\Lk(x)$ with the neighbors of
$x$.  The ascending link $\alk(x)$ is the full subcomplex spanned by
vertices $y$ such that $h(y)>h(x)$ and likewise the descending link
$\dlk(x)$ is the full subcomplex spanned by vertices $y$ such that
$h(y)<h(x)$.  These ascending and descending links play a similar role
to the ascending and descending manifolds in classical Morse theory.

If $X$ has one vertex, then all vertices of $\widetilde{X}$
have the same link, so we will write 
$\Lk(\widetilde{X})$, $\alk(\widetilde{X})$, and
$\dlk(\widetilde{X})$.  The link $\Lk(\widetilde{X}_\Lambda)$ has two
vertices $s^\pm$ for each generator $s$ of $A_\Lambda$; the ascending
link $\alk(\widetilde{X}_\Lambda)$ is spanned by the $s^+$'s, and the
descending link $\dlk(\widetilde{X}_\Lambda)$ is spanned by the $s^-$'s.
If $Y$ is the flag complex with 1-skeleton $\Lambda$, then
$\alk(\widetilde{X}_\Lambda)$ and $\dlk(\widetilde{X}_\Lambda)$ are isomorphic
to $Y$.

% more detail about combinatorial Morse theory here?

\subsection{Labeled oriented graph groups}

We can also construct groups using \emph{labelled oriented graphs}
(LOGs).  A LOG on a set $S$ is a directed
multigraph $\Gamma$ with vertex set $S$ and a labeling of the
edges given by $l:E(\Gamma)\to S$; loops and multiple edges are
allowed.  We say that $\Gamma$ presents the group:
$$B_\Gamma:=\langle S\mid i(e)^{l(e)}=t(e), e\in E(\Gamma)\rangle,$$
where the notation $a^b$ represents the conjugation $b^{-1}ab$ and
where $i:E(\Gamma)\to S$ and $t:E(\Gamma)\to S$ are the functions
taking an edge to its start and end.  Since each relation has length
4, the presentation 2-complex $X_{\Gamma}$ of ${B_\Gamma}$ is a
2-dimensional cube complex.

Note that although $i(e)=t(e)$ is possible, we may assume that $i(e)\ne l(e)\ne t(e)$
since otherwise we could contract such an edge without changing the group.  This
implies that $\Lk(X_{\Gamma})$ contains no loops or edges of the form $s^+s^-$.

As with RAAGs, we can apply Morse theory to LOG groups.
Let $h_{B_\Gamma}:{B_\Gamma}\to \Z$ be the homomorphism mapping
each $s\in S$ to $1$.  This homomorphism can be extended linearly
over each cell of $\widetilde{X}_{\Gamma}$ to get a map
$h_{X_\Gamma}:\widetilde{X}_{\Gamma}\to \R$.  Consider the level set
$L_{B_\Gamma}=h_{X_\Gamma}^{-1}(0)\subset \widetilde{X}_{\Gamma}$.  
As in the RAAG case, the group
$\ker h_{B_\Gamma}$ acts on $L_{B_\Gamma}$ vertex-transitively, so if 
$L_{B_\Gamma}$ is connected, then its 1-skeleton is a Cayley graph for 
$\ker h_{B_\Gamma}$.  Edges in $L_{B_\Gamma}$ are
diagonals of squares in $X_{\Gamma}$, so each orbit of squares labeled
$a^b=c$ contributes a generator that can be written as $cb^{-1}$
or $b^{-1}a$.  See Fig.~\ref{fig:morse}.

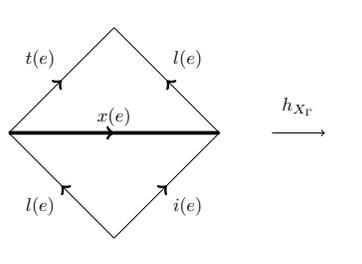
\begin{figure}[ht]
\begin{center}
\scalebox{.7}
{
\begin{tikzpicture}
\draw (-1,1) -- (-2,2);
\draw (-1,3) -- (0,4);
\draw (1,3) -- (0,4);
\draw (1,1) -- (2,2);

\draw[
    decoration={markings,mark=at position 1 with {\arrow[ultra thick]{>}}},
    postaction={decorate}
    ]
   (0,0) -- (-1,1);    
\draw[
    decoration={markings,mark=at position 1 with {\arrow[ultra thick]{>}}},
    postaction={decorate}
    ]
   (-2,2) -- (-1,3);    
\draw[
    decoration={markings,mark=at position 1 with {\arrow[ultra thick]{>}}},
    postaction={decorate}
    ]
   (2,2) -- (1,3);    
\draw[
    decoration={markings,mark=at position 1 with {\arrow[ultra thick]{>}}},
    postaction={decorate}
    ]
   (0,0) -- (1,1);    

\draw[line width=2pt,
    decoration={markings,mark=at position 1 with {\arrow[ultra thick]{>}}},
    postaction={decorate}
    ]
   (-2,2) -- (0,2);    

\draw[line width=2pt] (0,2) -- (2,2);

\draw[->] (3,2)--(4,2);
\draw(3.5,2.5) node{$h_{X_\Gamma}$};
\draw (5,-0.5) -- (5,4.5);
\fill[black] (5,0) circle (2pt);
\fill[black] (5,2) circle (2pt);
\fill[black] (5,4) circle (2pt);

\draw (-1.4,.6) node{$l(e)$};
\draw (1.4,.6) node{$i(e)$};
\draw (1.4,3.4) node{$l(e)$};
\draw (-1.4,3.4) node{$t(e)$};
\draw (0,2.3) node{$x(e)$};

\end{tikzpicture}
}
\caption{The map $h_{X_\Gamma}$ on a $2$-cell in $\widetilde{X}_\Gamma$.  The group element $x(e)$ is in $\ker h$.}
\label{fig:morse}
\end{center}
\end{figure}

As was the case with RAAGs, the link
$\Lk(\widetilde{X}_{\Gamma})$ has two vertices $s^\pm$ for each vertex $s$ of
$\Gamma$.  The ascending link $\alk(\widetilde{X}_{\Gamma})$ is the full
subcomplex of $\Lk(\widetilde{X}_{\Gamma})$ spanned by the $s^+$'s, and the descending
link $\dlk(\widetilde{X}_{\Gamma})$ is spanned by the $s^-$'s.
Brady showed:

\begin{theorem}\cite[Prop.~2.2.7]{BradyNonPos}\label{thm:LOGGs}
Suppose ${B_\Gamma}$ is a group presented by a LOG $\Gamma$ such that:
\begin{itemize}
\item the ascending and descending links $\alk({\widetilde{X}_\Gamma})$ and 
$\dlk({\widetilde{X}_\Gamma})$ are trees, and
\item the full link $\Lk({\widetilde{X}_\Gamma})$ has girth at least 4.
\end{itemize}
Then (1) $X_{\Gamma}$ is locally CAT(0), hence a $K(B_\Gamma,1)$; 
(2) the level set $L_{B_\Gamma}$ is a tree; and (3) $B_\Gamma$ is isomorphic to the 
free-by-cyclic group $F_n\rtimes_\phi \Z$, where $F_n\cong \ker h_{B_\Gamma}$.
\end{theorem}
In~\cite{BrGuLe}, Brady, Guralnik, and Lee used these groups to construct Stallings-type
examples of groups which are of type $F_2$ but not of type $F_3$ and
which have Dehn functions with prescribed polynomial or exponential
growth rates.

\section{Main Theorem}\label{sec:mainthm}

To understand our construction, first consider the problem of
constructing a space where the homological and homotopical filling
functions differ.  Suppose $W$ is a simply-connected space with large
Dehn function and $\alpha$ is a closed curve in $W$.  In order to
reduce the homological filling area but not the homotopical
filling area of $\alpha$, we could attach a 2-complex $Z$ to $\alpha$, in which
$\alpha$ is the boundary of a 2-chain, but not the boundary of a disk.
If $\pi_1(Z)/\langle \alpha\rangle=0$, the
resulting space is still simply connected.  By attaching copies of $Z$
to infinitely many closed curves, we can obtain a complex which has large
$\delta$ but small $\FA$.

Our construction will be based on a graph of groups with each vertex
labeled by one of two groups, $A$ and $Q$.  The first group, $A$, will
be a right-angled Artin group with a kernel $H_A$ that is $FP_2$ but not
finitely presented.  
This subgroup acts geometrically on a
space which has trivial $H_1$ and non-trivial $\pi_1$, which 
will provide the $Z$'s in the
construction.   

We define a \emph{Thompson complex} to be a connected, finite,
2-dimensional flag complex $Y$ whose fundamental group is a
simple group with an element of infinite order.  (The name comes from
the first known group with these properties, Thompson's group $T$.)  

Let $Y$ be a Thompson complex (for example, a triangulation of a
presentation complex for Thompson's group).  Note that since $\pi_1(Y)$ is simple, $H_1(Y)=0$, 
and every $g\ne 1\in\pi_1(Y)$ normally generates all of $\pi_1(Y)$.
Let $g\in \pi_1(Y)$
be an element of infinite order.  By gluing an annulus to $Y$, we may
assume that there is a path of length 4 in the $1$-skeleton of $Y$
which represents $g$.  We label the vertices of this path $a,u,s,v$, and label the rest
of the vertices of $Y$ by $y_1,\dots,y_d$.  Since $Y$ is flag, the
subcomplex spanned by $a,u,s,v$ must be a cycle of length 4.  
We will consider the RAAG $A_\Lambda$ where $\Lambda$ is the 1-skeleton of $Y$.

As we won't need to refer to $\Lambda$ explicitly, we drop it from the notation and
set $A=A_\Lambda$.  We denote the associated homomorphism by 
$h_A:A\to\Z$, its extension to a Morse function by $h_{X_A}:\widetilde X_A\to\R$,
the level set $h_{X_A}^{-1}(0)$ by $L_A$, etc.
By results of \cite{BestBrad}, the group $H_A=\ker h_A$ is $FP_2$ but not finitely presented.

The second group, $Q$, will be a product of a LOG group and a free group.  Suppose we are given a LOG $\Gamma'$ that satisfies the hypotheses of
Theorem~\ref{thm:LOGGs}.  We may form a new LOG $\Gamma$ by adding an
isolated vertex $s$ to $\Gamma'$, adding a loop connecting a vertex
$t$ of $\Gamma'$ to itself, and labeling the new edge by $s$.  This
corresponds to adding a generator $s$ and a relation $[s,t]=1$ to $B_\Gamma$.
We call a LOG $\Gamma$ obtained this way a \emph{special
  LOG} or SLOG, and the corresponding group a \emph{SLOG group}.  
  Note that $\Gamma$ still satisfies the hypotheses of Theorem~\ref{thm:LOGGs}.  
 
As with $\Lambda$, we will often omit $\Gamma$ from the notation when it
is easily understood.  For instance we will abbreviate $B_\Gamma$ by $B$, 
$X_\Gamma$ by $X_B$, $h_{B_\Gamma}$ by 
$h_B$, etc.

If $B$ is a SLOG group, then by Theorem \ref{thm:LOGGs}, it can be
written as a free-by-cyclic group $B=F_n\rtimes_\phi\Z$.  (The
notation $F_n$ indicates a rank $n$ free group; if we want to
emphasize a particular free basis $\{x_i\}$ we will write
$F_n(x_1,\ldots,x_n)$.)  We define
$\dist_{B}$ to be the distortion of $F_n$ inside $B$; precisely,
$$\dist_{B}(\ell)=\max_g \{ |g|_{F_n} \ : |g|_B\le\ell\},$$
where $|g|$ is the word length of $g$ in the subscripted group.  
In \cite{BrGuLe} there are constructions of SLOG groups $B=F_n\rtimes_\phi\Z$
with $\dist_B\sim e^\ell$, and also with $\dist_B\sim\ell^d$ for
all sufficiently large integers $d$.  The Bieri-Stallings double of $B$,
denoted $D=B \ast_{F_n} B$, has large Dehn function
resulting from this distortion; Bridson and Haefliger showed
\begin{theorem}[\cite{BridHaef}, Thm.\ III.$\Gamma$.6.20]\label{thm:bhLowerDouble}
  If $B$ and $D$ are as above, then 
  $$\dist_B(\ell)\preceq \delta_D(\ell).$$
\end{theorem}
The group $D$ will serve as $W$ in our construction; since its Dehn function is large, it has many curves which are difficult to fill by disks.  
By an embedding trick appearing in \cite{BaumslagBridsonMillerShort}, $D$ 
can be viewed as a subgroup of the product $Q:=B\times F_2$, and in fact
we will see that $D$ is the kernel of a map $h_Q:Q\to\Z$.

We will construct a finitely presented CAT(0) group $G$ as a graph product of
$Q$ with several copies of $A$.  The subgroup $H$ will be the kernel of a map 
$h:G\to\Z$, and $H$ will have the structure of a graph product of copies of $D$
and $H_A$.
We will show that attaching $H_A$ to $D$ does not affect $\delta$, but that the 
copies of $Y$ that lie in $L_A$ can be used to replace fillings by disks with more 
efficient fillings by chains.

\begin{theorem}\label{thm:mainThm}
Let $A$ be a RAAG based on a Thompson complex as described above,
and let $B=F_n\rtimes_\phi\Z$ be a SLOG group.  Then there exists a finitely presented CAT(0) group 
$G$ containing $A$ and $B$ such that the
homomorphisms $h_A:A\to\Z$ and $h_B:B\to\Z$ extend to 
$h:G\to\Z$ and such that $H=\ker(h)$ is finitely presented and satisfies:
\begin{align*}
\ \FA_H (\ell) &\ \preceq \ \ell^5 \\
\delta_H(\ell)&\ \succeq\ \dist_B(\ell).
%&\ \preceq \ \ell^3\cdot\dist_{\phi}(\ell).
\end{align*}
\end{theorem}

Using the examples of SLOG groups constructed in \cite{BrGuLe}, this implies
Theorem \ref{thm:THM}.

\section{Constructing $G$ and $H$}

In this section we construct the groups $G$ and $H$ of
Theorem~\ref{thm:mainThm}.  
The construction is similar to the perturbed
RAAGs in \cite{BrGuLe}, but we glue several RAAGs (rather than just one)
to a free-by-cyclic group.  

Throughout this paper, if $g$ is a group element, $\bar{g}$ will represent its inverse.

\subsection{The SLOG piece}
Let $B=B_\Gamma=F_n\rtimes_\phi\Z$ be as in Theorem~\ref{thm:mainThm}.
The first step of the construction is to use
$B$ to construct a group $D\triangleleft Q\cong B\times F_2$ with large Dehn function.
The group $Q$ will contain several copies of the group 
$F_2\times F_2$, and we will attach 
RAAGs $A_i$ to $Q$ along some of these groups. 
The result
of this gluing will be $G$.

Since $\Gamma$ is a SLOG, it contains an isolated vertex $s$ which is the label of a 
single loop in $\Gamma$.  Call the vertex of that loop $t$.  
Call the rest of the vertices $\{a_1,\ldots,a_{n-1}\}$.  We have 
two presentations for $B$, namely the SLOG presentation with 
generating set $\{s,t,a_i\}$, and a free-by-cyclic presentation.  For the latter 
we may take $\{x_i,t \mid 1\le i \le n\}$ as a generating set, where 
$x_i=a_i\bar t$ for $1\le i < n$ and $x_n=s\bar t$.  (See Fig.~\ref{fig:morse}.)
Thus $B=F_n(x_1,\dots,x_n)\rtimes_\phi \Z$.

Let $D$ be the double $D=B \ast_{F_n} B$ where the $F_n$ is generated
by the $x_i$.  Theorem~\ref{thm:bhLowerDouble} implies that $D$ has
Dehn function at least as large as $\dist_B$.  By
\cite{BaumslagBridsonMillerShort}, $D$ is isomorphic to the subgroup
$$D\cong F_n(x_1,\dots,x_n) \rtimes_{(\phi,\phi)} F_2(t\bar u,t\bar v)$$
of the group $Q=B \times F(u,v)$.  Furthermore, if
$h_Q:Q\to\Z$ is the group homomorphism taking the elements
$a_i,s,t,u,v$ to $1\in\Z$, then the kernel of $h_Q$ is precisely $D$.

Since $\Gamma$ is a SLOG, the group $B=B_\Gamma$
contains many copies of $F_2$.  Recall that the presentation 2-complex $X_B$ of
$B$ is a locally CAT(0) 2-dimensional cube complex and thus a
$K(B,1)$.  For any $i$, consider the subgroup of $B$
generated by $a_i$ and $s$.  
%Each generator $g$ corresponds to two
%vertices $g^\pm$ in the link of $X_{B_\Gamma}$, and it
It is easy to
check that any two of the vertices $a_i^{\pm}$ and $s^{\pm}$ 
in the link of $X_{B}$ are
separated by distance at least 2 in $\Lk(X_B)$.  Consequently,
the Cayley graph of the subgroup generated by $a_i$ and $s$ is convexly embedded in 
$\widetilde{X}_{B}$.  It is thus a copy of $F_2$.%, the free group of rank $2$.

Define $X_Q=X_{B}\times R_2$, where $R_2$ is a
wedge of two circles, so that $X_Q$ is a $K(Q,1)$.  This is locally
CAT(0), and by the argument above, for all $i$, the subgroup generated by
$a_i,s,u,v$ is a convexly embedded copy of $F_2\times F_2$.  

\subsection{Attaching the RAAG pieces}
% Higman-Thompson?
Let $A=A_\Lambda$ be a RAAG constructed from a Thompson complex $Y$ as in 
Sec.~\ref{sec:mainthm}.  Thus $g\in\pi_1(Y)$ has infinite order
and is represented by a path $ausv$ in the defining graph $\Lambda$ for $A$.
Let $X_A$ be
the Salvetti complex of $A$.  Since $a,u,s$, and $v$
span a square in $Y$, $A$ has a convex subgroup isomorphic to $F_2\times F_2$,
generated by $a,u,s,v$.  Let $E:=F_2\times F_2$.

We form the group $G$ by gluing copies of $A$ to $Q$ along copies of $E$.
Specifically, consider a graph of groups with vertex groups $Q,A_1,\dots,A_{n-1}$, where $A_i=A$, 
and with each vertex $A_i$ connected to $Q$ to by an edge.  As noted above, for each $i$, the elements $a_i,s,u,v\in Q$ generate a copy of $E$; denote this copy by $E_i$.  We identify $E_i$ with the copy of $E$ in $A_i$ by $a_i\leftrightarrow a$,
$u\leftrightarrow u$, $s\leftrightarrow s$, $v\leftrightarrow v$.  Let
$G$ be the fundamental group of this graph of groups.  
This is a group generated by 
$$\{a_{1},\dots,a_{n-1},s,t,u,v,y^i_j\} \quad
i=1,\dots,n-1,\ j=1,\dots,d.$$  
We can define subgroups $Q$, $E_i$, and $A_i$, where $E_i\cong
E$, $A_i\cong A$, and 
$$Q=B_\Gamma\times F_2(u,v)=\langle
a_{1},\dots,a_{n-1},s,t,u,v\rangle,$$
$$A_i = \langle a_i,s,u,v,y^i_{1},\dots, y^i_{d}\rangle,$$
$$E_i=A_i\cap Q=F_2(a_i,s)\times F_2(u,v).$$
The
homomorphisms $h_Q:Q\to \Z$ and $h_A:A\to \Z$ agree on the edge
groups, so we can extend them to a function $h:G\to \Z$.

Let $H=\ker h$.

\section{Finite presentability}\label{sec:topology}

In this section, we will construct a space on which $H$ acts and
consider its topology.  This will let us prove that $H$ is finitely
presented and will help us bound the Dehn
functions of $H$.  

We can realize the above construction of $G$ geometrically to construct a
$K(G,1)$ as follows.  Let $X_E:=R_2\times R_2$, where $R_2$ is the wedge of two
circles; this is a $K(E,1)$, and each edge group corresponds to copies of $X_E$ in $X_A$ and $X_Q$.  Each of these copies of $X_E$ is convex, so we can
glue $n-1$ copies of $X_A$ to $X_Q$ along the $X_E$'s to obtain a
locally CAT(0) cube complex which we call $X_G$.  This is a $K(G,1)$,
and the 1-skeleton of its universal cover $\widetilde{X}_G$ is a Cayley graph of $G$.
In particular

\begin{lemma}
The group $G$ is CAT(0).
\end{lemma}
% The following lemma summarizes the properties of the groups
% constructed above:
% \begin{lemma}
%   \begin{enumerate}
%   \item 
%     $G$ is generated by 
%     $$S=\{a_{1},\dots,a_{n-1},s,t,u,v,y_{i,j}\} \quad i=1,\dots,n-1,
%     j=1,\dots,d.$$
%     The Cayley graph of $G$ with respect to this generating set is the
%     1-skeleton of a CAT(0) cube complex $\widetilde{X}_G$.
%   \item $B_\Gamma\cong \langle a_{1},\dots,a_{n-1},s,t\rangle$.  If we let $x_i=a_i\bar t$ for $1\le i < n$ and $x_n=s\bar t$, then 
%     $B_\Gamma=F_{n}(x_1,\dots,x_{n})\rtimes \langle t\rangle$.  Furthermore, $Q:=\langle B_\Gamma,u,v\rangle=B_\Gamma\times F_2(u,v)$.  
%   \item For each $i$, the group $A_i = \langle a_i,s,u,v,y_{i,j}\rangle$ is a right-angled Artin group with defining graph the 1-skeleton of $Y$.
%   \item For each $i$, the intersection $E_i:=A_i\cap Q$ is generated by $a_i,s,u,$ and $v$, and is a convex subgroup of $G$ which is isomorphic to $F_2(a_i,s)\times F_2(u,v)$.  Consequently, $G$ is a graph product of the $A_i$ and $Q$, with edge groups isomorphic to $F_2\times F_2$ connecting $Q$ to each $A_i$.
%   \end{enumerate}
% \end{lemma}

Now, $H$ is the kernel of the homomorphism $h:G\to \Z$.  As before,
the vertices of $\widetilde X_G$ are in correspondence with the elements of $G$, 
so by viewing $h$ as a function on the vertices of $\widetilde X_G$, we may
extend $h$ linearly over cubes to obtain a Morse function $h:\widetilde X_G\to\R$.
Let $L_G=h^{-1}(0)$.  

Since $h$ cuts cubes of $L_G$
``diagonally'', $L_G$ is a polyhedral 2-complex whose cells are slices of
cubes.  The subgroup $H$ acts freely on $L_G$, and
since the vertices of $L_G$ are in $1$-to-$1$ correspondence with the
elements of $H$, the action is cocompact and thus geometric.  We will show:
\begin{lemma}
  $L_G$ is simply-connected and thus $H=\ker(h)\subset G$ is finitely
  presented.
\end{lemma}
\begin{proof}
  Since $\widetilde{X}_G$ is contractible and $h$ is a Morse function on
  $\widetilde X_G$, Theorem~4.1 of \cite{BestBrad} implies
  that it is enough to show that the ascending and descending links of
  any vertex in $\widetilde{X}_G$ are simply connected.  Since $X_G$ has
  only one vertex, it is enough to show this for that vertex.  

  Since the 1-skeleton of $\widetilde{X}_G$ is the Cayley graph of $G$, we
  can label the vertices of $\Lk(X_G)$ by $g^\pm$, where $g$ ranges
  over the generating set $S$.  The link $\Lk(X_G)$ of the vertex of
  $X_G$ is obtained by gluing $\Lk(X_Q)$ and the various $\Lk(X_{A_i})$.
  For each $1\le i \le n-1$, the
  links $\Lk(X_Q)$ and $\Lk(X_{A_i})$ each contain a subcomplex with
  vertices $a^\pm_i$, $u^\pm$, $s^\pm$, and $v^{\pm}$, and gluing the
  links along these subcomplexes gives $\Lk(X_G)$.

  Likewise, we can form $\alk(X_G)$ by gluing $\alk(X_{A_i})$, $1\le i\le
  n-1$, to $\alk(X_Q)$ along subcomplexes $S_i$ spanned by
  $a^+_i,u^+,s^+,$ and $v^+$.  We claim that $\alk(X_G)$ is
  simply-connected.  Since $\alk(X_{B_\Gamma})$ is a tree by
  hypothesis, $\alk(X_{Q})$ is the suspension of a tree (with
  suspension points $u^+$ and $v^+$), and thus it is
  simply-connected.  Since $A_i$ is a right-angled Artin group with
  defining complex $Y$, each $\alk(X_{A_i})$ is isomorphic to $Y$, and
  each $S_i$ is a square such that the normal closure of $\pi_1(S_i)$
  in $\pi_1(\alk(X_{A_i}))=\pi_1(Y)$ is all of $\pi_1(Y)$.  By the
  Seifert-van~Kampen theorem,
  $$\pi_1(\alk(X_{G}))=(\pi_1(\alk(X_1))\ast \dots\ast
  \pi_1(\alk(X_{n-1})))/\langle
  \pi_1(S_1),\dots,\pi_1(S_{n-1})\rangle=0,$$ so the ascending link is
  simply-connected.  The same argument with $+$'s changed to $-$'s
  shows that the descending link is also simply connected, so $L_G$ is
  simply connected.  Thus $H=\pi_1(L_G/H)$ and $H$ is finitely presented.
\end{proof}

For an alternate description of $H$, 
recall that the group $G$ is the fundamental group of a graph of groups, with one
vertex labeled $Q$ connected to $n-1$ vertices labeled $A_i$ by edges
labeled $E_i$.  Since $H\subset G$, $G$ induces a graph of groups
structure on $H$.  Indeed, $G$ acts on a tree $T$ whose vertices
correspond to the cosets of $Q$ and $A_i$, whose edges correspond to
cosets of $E_i$, and whose quotient $G\backslash T$ is a star with $n-1$
edges.  We can restrict the action of $G$ on $T$ to an action of
$H$, and since any coset of $Q$, $A_i$, or $E_i$ has nontrivial
intersection with $H$, the orbit of any vertex or edge under $H$ is
the same as its orbit under $G$.  Therefore $H$ acts on $T$ with
vertex stabilizers conjugate to $H_Q:=H\cap Q$ and $H_{A_i}:=H\cap
A_i$, edge stabilizers conjugate to $H_{E_i}:=H\cap {E_i}$, and
quotient $H\backslash T=G\backslash T$.  This shows that $H$ is the
fundamental group of the graph of groups with central vertex labeled
$H_Q$ connected to $n-1$ vertices labeled $H_{A_i}$ by edges labeled
$H_{E_i}$.

The level set $L_G:=h^{-1}(0)\subset \widetilde{X}_G$, however, is
\emph{not} the universal cover of a corresponding graph of spaces.  To
describe $L_G$, we define $L_Q:=L_G\cap \widetilde{X}_Q$,
$L_{A_i}:=L_G \cap \widetilde{X}_{A_i}$, $L_{E_i}:=L_G\cap
\widetilde{X}_{E_i}$.  These level sets have geometric actions by
$H_Q$, $H_{A_i}$, and $H_{E_i}$ respectively.  We can write the
quotient $L_G/H$ as $L_Q/H_Q$ with the $L_{A_i}/H_{A_i}$ attached
along copies of $L_{E_i}/H_{E_i}$, but since $L_{A_i}$ and $L_{E_i}$
are not simply-connected, $\pi_1(L_{A_i}/H_{A_i})\ne A_i$ and
$\pi_1(L_{E_i}/H_{E_i})\ne E_i$; this is a graph of spaces for a
different graph of groups.

The fact that $L_{A_i}$ and $L_{E_i}$ are not simply-connected will be
important in the rest of the paper, so we will go into some more
detail.  All of the $L_{A_i}$'s and all of the $L_{E_i}$'s are isometric, so
when $i$ is unimportant, we will denote them by $L_A$ and $L_E$.  To
understand the topology of $L_A$ and $L_E$, consider them as subsets
of $\widetilde X_A$.  By \cite{BestBrad}, $L_A$ is a union of scaled
copies of $Y$, indexed by vertices in $\widetilde X_A\setminus L_A$.
Likewise, $L_E$ is composed of scaled copies of a square, which we
denote $\lozenge$, indexed by vertices in $\widetilde X_E\setminus
L_E$.  Translating $L_E$ by elements of $H_A$ gives infinitely many
disjoint copies of $L_E$ inside $L_A$.

By Theorem~8.6 of \cite{BestBrad}, $L_A$ is homotopy equivalent to an
infinite wedge sum of copies of $Y$ and $L_E$ is homotopy equivalent
to an infinite wedge sum of copies of $\lozenge$, so $L_A$ and $L_E$
have infinitely generated $\pi_1$.  Generators of $\pi_1(L_A)$ can
be filled in two ways.  First, each generator can be freely homotoped
into some copy of $L_E$.  Each copy of $L_E$ is contained in some
$L_Q$, and since $L_Q$ is simply connected, each generator of
$\pi_1(L_A)$ is filled by a disk in one of the copies of $L_Q$.
Second, though $\pi_1(L_A)$ is infinitely generated, $H_1(L_A)$ is
trivial, so any curve in $L_A$ can be filled by some 2-chain entirely
inside $L_A$.  Our goal in the rest of this paper is to use these two
types of fillings to show that the homological and homotopical Dehn
functions of $L_G$ are different.

\section{Upper bound on the homological Dehn function}

In this section we prove the following.
\begin{proposition}\label{logupper}
With $H$ as above, we have $\FA_H(\ell)\preceq \ell^5$.
\end{proposition}

\begin{proof} We show that any $1$-cycle in $L_G$ of mass at most $\ell$ can be filled 
by a $2$-chain of mass $\preceq \ell^5$.  Since $\ell^5$ is a super-additive function, it is enough to prove this for loops in $L_G^{(1)}$.  

 $L_G$ consists of copies of $L_A$ and $L_Q$ glued together along copies of $L_E$.  We first show how to homologically fill loops that lie in a single copy of $L_A$ or $L_Q$, and then use these fillings to fill arbitrary loops.  
Note that each 1-cell of $L_G$ is a diagonal of some square in
$\widetilde{X}_G$, so each 1-cell corresponds to a product $x\bar y$ where $x$ and $y$ are (certain) generators of $G$.

Consider a loop $\alpha$ of length $\ell$ that lies in a copy of $L_{A}$.  Recall that $L_{ A}$ is a level set of the Morse function $h_A:\widetilde{X}_A \to \R$.  Since $\widetilde{X}_A$ is CAT(0), there exists a 2-chain $\beta$ with boundary equal to $\alpha$ and 
mass $\preceq \ell^2$. Further, $\beta$ lies in $h_A^{-1}[-c\ell, c\ell]$ for some $c>0$ (\cite{Wenger08}, cf. Prop 2.2 in \cite{ABDDY}).   
We will use $\beta$ to produce a filling of $\alpha$ in $L_A$ using a pushing map as in~\cite{ABDDY}. 

Let $Z$ be the space obtained by deleting open neighborhoods of the vertices of $X_A$ outside $L_A$, with the induced cell structure.  So $Z = X_A \setminus \cup_{v \notin L_A} B_{1/4}^{\circ}(v)$.  
By Theorem~4.2 in \cite{ABDDY}, there is a $H_A$-equivariant Lipschitz retraction (pushing map) $\mathcal{Q}:Z \to L_A$ such that the Lipschitz constant grows linearly with distance from $L_A$.  Furthermore, if $S_v=\partial B_{1/4}^{\circ}(v)$ is the boundary of one of the deleted neighborhoods, the image of $S_v$ is a copy of $Y$ with metric scaled by a factor of $h_A(v)$.  In particular, if $\gamma$ is a 1-cycle in $S_v$ of length $\ell(\gamma)$, then $\mathcal{Q}_\sharp(\gamma)$ is a 1-cycle in a scaled copy $Y_v$ of $Y$.  Since $H_1(Y)=0$ and $Y$ is compact, the corresponding 1-cycle in $Y$ has homological filling area $\preceq \ell(\gamma)$.  Therefore the original cycle, $\mathcal{Q}_\sharp(\gamma)$, has homological filling area $\preceq \ell(\gamma) h_A(v)^2$.

Consider the restriction $\beta'$ of $\beta$ to $Z$; this is a 2-chain in $Z$, and $\partial\beta'=\alpha+\sum_{v\in \supp{\beta}\setminus L_A}\gamma_v$, where $\gamma_v$ is a 1-cycle in $S_v$.  We can construct a filling $\beta''$ of $\alpha$ in $L_A$ by combining the image $\mathcal{Q}_\sharp(\beta')$ with fillings of each of the $\mathcal{Q}_\sharp(\gamma_v)$'s.  Since $\beta'$ lies in $h_A^{-1}[-c\ell, c\ell]$, the restriction of $\mathcal Q$ to $\supp{\beta'}$ has $O(\ell)$ Lipschitz constant, and
$$\mass(\beta'')\preceq \ell^2\mass(\beta)+\ell^2 \sum_v \ell(\gamma_v)\preceq \ell^4+\ell^2 \sum_v \ell(\gamma_v).$$
Each 2-cell of $\beta$ contributes at most four 1-cells to the $\gamma_v$'s, so $\sum_v \ell(\gamma_v)\preceq \ell^2$, and $\mass(\beta'')\preceq \ell^4$, as desired.

Next we produce a quartic mass filling of any loop that lies entirely in a copy of $L_Q$.  
Such a loop $\alpha$ is labeled by generators of 
$$D\cong F_n(x_1,\dots,x_n) \rtimes_{(\phi,\phi)} F_2(t\bar u,t\bar v)$$
where $x_i=a_i\bar t$ for $1\le i < n$ and
$x_n=s\bar t$.
In this section we will use the notation $t_1=t\bar u$ and $t_2=t\bar v$. 
Let $w$ denote the word labeling $\alpha$, where
$$w=w_1\dots w_\ell.$$ 
Here the $w_i$ are generators and $w$ represents the identity.  

Let $\vartheta:F(x_1,\dots,x_n )\rtimes F_2(t_1, t_2) \to
F(x_1,\dots,x_n ) \rtimes \langle t_1 \rangle$ send $t_2$ to $t_1$ and
each other generator to itself.  
The word $\vartheta(w)$ lies in a free-by-cyclic
group which is CAT(0) and therefore has quadratic Dehn function.
Thus to fill our loop with quartic mass, it will be enough to reduce $w$ to
$\vartheta(w)$ in such a way that the reduction takes quartic mass.

We will achieve this reduction by first decomposing $w$ into subwords as follows.  
Let $p:F(x_1,\dots,x_n)\rtimes F_2(t_1, t_2)\to F_2(t_1, t_2)$ be the projection map.
Define $$w(i)=w_1\dots w_i$$ so that $w(0)=w(\ell)=1$.  
Now decompose $w$ as follows:
$$w=[p(w(0))w_1p(w(1))^{-1}][p(w(1))w_2p(w(2))^{-1}]\dots
[p(w(\ell-1))w_\ell p(w(\ell))^{-1}].$$ 
We can reduce $w$ to $\vartheta(w)$ by reducing each subword in this decomposition to its image under $\vartheta$.  
If $w_j=t_1^{\pm1}, t_2^{\pm1}$, then $p(w(j-1))w_{j}p(w(j))^{-1}$ is freely
equal to the identity, and no reduction is necessary.  
 Otherwise, $p(w(j-1))=p(w(j)) \in F(t_1, t_2)$.
Thus it suffices to reduce words of the form $g x_i \bar g$, with $g\in F(t_1,t_2)$, to 
$\vartheta(g x_i \bar g)$, or equivalently, to
fill loops with labels 
of the form $v=g x_i \bar g \vartheta(g \bar x_i \bar g)$.

Write $g=t_{r_1}^{d_1} \dots t_{r_m}^{d_m}$ where $r_i$ alternates
between $1$ and $2$ and $d_i\in \Z\setminus \{0\}$.  We proceed by
induction on $m$.

If $m=1$ and $g=t_1^d$ then there is nothing to do.  If $g=t_2^d$, then $v=
t_2^dx_it_2^{-d}t_1^d \bar x_it_1^{-d}$ can be written as the sum of four 1-cycles as shown in Fig.~\ref{fig:hom}.   
Writing $s_1=s\bar u$ and $s_2=s\bar v$, the words labeling the $1$-cycles are 
$t_2^ds_2^{-d}s_1^dt_1^{-d}$, 
$s_2^d (a_i\bar v)^{-d} (a_i\bar u)^d s_1^{-d}$, 
$ (a_i\bar v)^{d+1} s_2^{-(d+1)}   s_1^{d+1} (a_i\bar u)^{-(d+1)}$, and 
$s_2^{d+1}t_2^{-(d+1)}t_1^{(d+1)}s_1^{-(d+1)}$.
The first and the last are words representing the identity in the CAT(0) group 
$\langle t \rangle \times \langle s \rangle \times F(u,v)$, and so can be filled with quadratic mass.  
The middle two are generators of $\pi_1(L_A)$ and can be filled in $L_A$ with a scaled copy of $Y$ with quadratic mass.   These four fillings fit together to give a filling of $v$ with quadratic mass.  

\begin{figure}
\centering
\def\svgwidth{2.5in} 
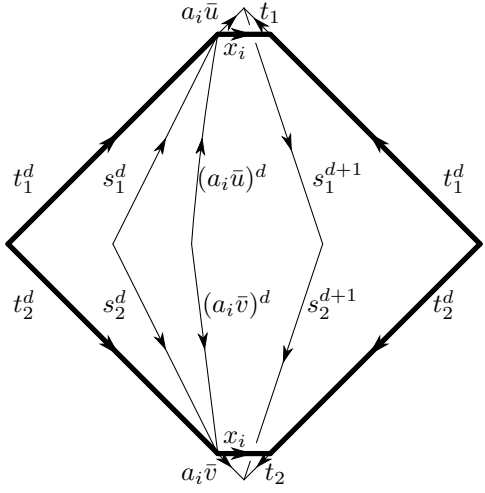 
\caption{\label{fig:hom} A homological filling of $v$.  The curve $v$ (thick line) is a sum of four 1-cycles (thin lines).}
\end{figure}

If $m>1$, then $g=g_0 t_{r_m}^{d_m}$.  Let $r=r_m$ and $d=d_m$ and let $g'=g_0 t_{3-r}^{d}$ (and note that
$\vartheta(g)=\vartheta(g')$).  As in the $m=1$ case, we can reduce
$t_r^dx_i
t_r^{-d}$ to $t_{3-r}^dx_i t_{3-r}^{-d}$ using quadratic area.   
This 
immediately lets us reduce
$$g x_i \bar g \vartheta(g \bar x_i \bar g)$$
to $$g'x_i g'^{-1} \vartheta(g' x_i^{-1} g'^{-1}).$$

Since we use $m$ steps, and $m\le \ell(g) \le \ell$, it takes area $\preceq \ell^3$
(and linear genus) to reduce $g x_i \bar g$ to $\vartheta(g \bar x_i \bar g)$.

Since each of the $\ell$ subwords in the decomposition of $w$ above can be reduced to its image under $\vartheta$ using mass $\preceq \ell^3$, the word $w$ can be reduced to $\vartheta(w)$ with mass $\preceq \ell^4$.  

% Finally, we consider curves that travel through multiple copies of
% $L_{A_i}$ and $L_Q$.  Let $\alpha$ be a path of edges in $L^{(1)}$ of
% length $\ell$.  We can partition $\alpha$ into arcs $\alpha_1,\dots,
% \alpha_n$, each of which lies in a single copy of $L_Q$ or $L_{A_i}$.
% By standard arguments (i.e., the normal form theorem for graphs of
% groups), there is some such subarc whose endpoints lie in a single copy of $L_E$.  

% consider a subarc of $\alpha$ which is a component 
% of the intersection of $\alpha$ with a copy of $L_Q$ (or $L_{A_i}$) corresponding to some coset of $H_Q$ (or $H_{A_i}
% $).  This arc enters and leaves the copy of $L_Q$ or $L_{A_i}$ via a pair of vertices in a copy of $L_{E_i}$.  
% Now $L_{E_i}$ is connected and undistorted in $L_G$.  (The latter is because $L_{E_i}$ is undistorted in 
% $X_{E_i}$, which is convex in $L_G$.)   So the two endpoints of the arc an be connected by an undistorted path in 
% $L_{E_i}$, to produce a loop with $O(\ell)$ length that lies entirely in a copy of $L_Q$ or $L_{A_i}$.  Proceeding 
% inductively we can write $\alpha$ as a sum of at most $\ell$ loops (1-cycles) such that each loop has $O(\ell)$ length  
% and lies in a single copy of $L_Q$ or $L_{A_i}$.  
% Since each such loop has a filling of mass $\preceq \ell^4$, this produces a chain with mass $\preceq \ell^5$ which 
% fills $\alpha$.

Finally, we consider curves that travel through multiple copies of $L_{A_i}$ and $L_Q$.  We will need to make arguments based on the graph product decomposition of $H$, so it will be helpful to have a slightly different complex $L$ on which $H$ acts.  We construct $L$ by ``stretching'' each copy of $L_E$ in $L_G$ into a product $L_E\times [0,1]$.  
Let $Z$ be the complex obtained by gluing 
$X_Q$ and $n-1$ copies of $X_A$ to $n-1$ copies of $X_E\times [0,1]$ according to 
the graph product decomposition of $G$ and let $\widetilde{Z}$ be the universal cover of $Z$.
Then the homotopy equivalence $p:Z\to X_G$ which collapses each copy of 
$X_E\times [0,1]$ to a copy of $X_E$ lifts to a homotopy equivalence $\widetilde{p}:\widetilde{Z}\to \widetilde{X}_G$.
Then $H=\ker(h)\subset G$ acts geometrically on the level set 
$L=(h\circ \widetilde{p})^{-1}(0)$, because $H$ acts geometrically on $L_G$.
The following lemma describes the structure of $L$.
\begin{lemma}
  The level set $L$ intersects each vertex space of $\widetilde{Z}$ of the form
  $\widetilde X_Q$ in a copy of $L_Q$.  Likewise, it intersects each
  vertex space $\widetilde X_A$ in a copy of $L_A$, and
  each edge space $\widetilde X_E\times [0,1]$ in a copy of
  $L_E\times[0,1]$.  
\end{lemma}

Each edge in $L^{(1)}$ either lies in a copy of $L_Q$, lies in a copy
of $L_{A_i}$ for some $i$, or crosses from $L_E\times \{0\}$ to
$L_E\times \{1\}$, so we can classify the edges as $Q$-edges,
$A$-edges, or $E$-edges.  Consider a path of edges in $L^{(1)}$ of
length $\ell$ and call it $\alpha$.  

By standard arguments (i.e., the normal form theorem for graphs of
groups), $\alpha$ must have an ``innermost piece'', i.e., a subpath
which enters a copy of $L_{A_i}$ or $L_Q$ through a copy of $L_E$,
then leaves through the same $L_E$.  We write this as $t \gamma t'$,
where $t$ and $t'$ lie in the same copy of $L_E\times [0,1]$ and where
$\gamma$ is either a path of $A$-edges or a path of $Q$-edges.

Without loss of generality, suppose that the endpoints of $\gamma$ lie
in $L_E\times \{0\}$.  Call them $(v,0)$ and $(w,0)$, and let
$\gamma'$ be a geodesic in $L_E\times\{0\}$ from $v$ to $w$.  Then $p$
and $\gamma'$ form a loop $\theta$, and since they lie in the union of
a copy of $L_E\times[0,1]$ and a copy of $L_Q$ or $L_{A_i}$, there is
a 2-chain filling $\theta$ whose mass is $\preceq
(\ell+\ell(\gamma'))^4+\ell(\gamma')$.  But $L_{E}$ is undistorted in
$L$, since $L_{E}$ is undistorted in $X_{E}$ and $X_E$ is convex in
$X_G$, so $\ell(\gamma')\preceq \ell$, and the
filling area of $\theta$ is $\preceq \ell^4$.

Repeating this process for the loop $\alpha-\theta$ inductively, we
obtain a filling of $\alpha$.  Each time we repeat the process, the
number of $E$-edges in $\alpha$ decreases by 2, and each time we
repeat the process we use a filling of mass $\preceq \ell^4$, so the
total filling area of $\alpha$ is $\preceq \ell^5$.
\end{proof}

\section{Lower bound on homotopical Dehn function}

Recall the situation we are in:  the group $G$ is a graph product of groups 
$Q$ and (copies of) $A$ along edge groups $E$, and $\widetilde{X}_G,\widetilde X_Q,\widetilde X_A,\widetilde X_E$ are 
contractible spaces on which respectively $G,Q,A,E$ act geometrically.  
Meanwhile $H< G$ is a graph product of $D$ (the double of a SLOG group $B$) and 
copies of $H_A$ (the Bestvina-Brady group associated to the RAAG $A$).  
In this section we prove the following.

\begin{theorem}\label{thm:lowerDehn}
With $G$ (hence $H,D,Q,A,B$) as above, we have $$\delta_H \succeq \dist_B,$$
where $\dist_B$ is as in Sec.~\ref{sec:mainthm}.
\end{theorem}

In order to prove this, we will need the following refinement of Thm.~\ref{thm:bhLowerDouble}:
\begin{lemma}\label{lem:lowerDoubleSize}
  For all $\ell>0$, there is a curve $\gamma:S^1\to L_Q$ of length
  $\sim \ell$ such that if $\beta\in C_2(L_Q)$ is a filling of $\gamma$, then 
  $$\area \supp \beta \succeq \dist_{B}(\ell),$$
  where $\supp \beta$ is the support of $\beta$.
\end{lemma}
We can take $\gamma$ to be the curve $w_n^{-1}w_n'$ used in the proof
of Thm.\ III.$\Gamma$.6.20 in \cite{BridHaef}; the proof in
\cite{BridHaef} shows that the image of any disk filling $\gamma$ has
to have area $\succeq \dist_{B}(\ell)$, but the same bound applies to
any chain filling $\gamma$ as well. 

\begin{proof}[Proof of Theorem~\ref{thm:lowerDehn}]
  Let $L$ be the complex constructed in the previous section, on which
  $H$ acts; this is made up of copies of $L_A$ and $L_Q$, joined along
  copies of $L_E\times[0,1]$.  The translates of the set
  $L_E\times\{1/2\}$ separate $L$ into infinitely many components, each
  of which is either a copy of $L_A$ glued to copies of $L_E\times
  [0,1/2]$ or a copy of $L_Q$ glued to copies of $L_E\times [1/2,1]$.
  Let $L_A'$ and $L_Q'$ be complexes isometric to each type of piece.
  We will refer to the union of the copies of $L_E\times\{1/2\}$ that lie in $L_A'$
  and $L_Q'$ as $\partial L_A'$ and $\partial L_Q'$.

  Fix a ``root'' copy $L_0$ of $L_Q'$ in $L$.  Let $\beta:D^2\to L_0$
  be as in Lem.~\ref{lem:lowerDoubleSize} and let $\gamma$ be its
  boundary.  Let $\tau:D^2\to L$ be a filling of $\gamma$ in $L$ and
  let $[\gamma]\in C_1(L_0;\R)$ be the fundamental class of $\gamma$.
  We will show that there is a chain $\D\in C_2(L_0;\R)$ which fills
  $[\gamma]$ and is supported on $L_0\cap \tau(D^2)$.  Then
  $\D-[\beta]$ is a 2-cycle in $L_0$, but since $L_0$ is 2-dimensional
  and contractible, this implies that $\D=[\beta]$.  Therefore,
  $\tau(D^2)$ contains $\beta(D^2)$, and 
  $$\delta_D(\ell)\sim \area \beta \le \area \tau$$
  as desired.

  First, we use the structure of $L$ to break a disk in $L$ into
  punctured disks whose images lie in copies of $L_Q'$ and $L_A'$.
  Homotope $\tau$ so that it is transverse to each copy of
  $L_E\times\{1/2\}$.  The preimages of the $L_E\times\{1/2\}$'s then
  divide $D^2$ into \emph{pieces} $M_1,\dots, M_n\subset D^2$, and for
  each $i$, $\tau|_{M_i}$ is a punctured disk in a copy of $L_A'$ or
  $L_Q'$.  Each $M_i$ has a distinguished boundary component which is
  homotopic to $\partial D^2$ in $D^2 \setminus M_i$, and we call that
  component the \emph{outer boundary} of $M_i$, denoted $\partial_oM_i$;
  we call the other boundary curves \emph{inner boundaries}.  Each
  boundary curve of $M_i$ either coincides with $\partial D^2$ or lies
  in the preimage of one of the $L_E\times\{1/2\}$'s.

  Next, we claim that if $M$ is a punctured disk in $L_A'$ with boundary in
  $\partial L_A'$, the topology of $L_A$ places strong
  restrictions on $M$.  More precisely,
  \begin{lemma}\label{lem:LaPuncDiscs}
    Suppose that $f:M\to L_A'$ is a punctured disk in $L_A'$ with boundary in
    $\partial L_A'$.  Suppose that $M$ has a distinguished boundary
    component $\partial_o M$ and other boundary components 
    $\partial_1M\dots \partial_mM$.  If $[\partial_iM]\in C_1(\partial
    L_A';\R)$ is the fundamental class of $\partial_iM$, then there
    are $a_i\in \R$ such that
    $$f_\sharp[\partial_oM]=\sum_{i=1}^m a_i f_\sharp[\partial_iM].$$
    Furthermore, we may assume that $a_i\ne 0$ only if $f(\partial_i
    M)$ and $f(\partial_oM)$ lie in the same copy of
    $L_E\times\{1/2\}$.
  \end{lemma}
  
  We will prove this lemma at the end of the section, after we use it
  to construct $\D$.  %To do this we proceed by induction.
  Let $M_i$ be one of the pieces of $D^2$, and suppose that $\tau$
  takes $\partial_oM_i$ to a curve in $L_0$.  (For example, take $M_i$
  such that $\partial_o M_i=\partial D^2$.)  Then $\tau(M_i)$ is
  either a punctured disk in $L_0$ or a punctured disk in a copy of
  $L_A'$ which neighbors $L_0$.  Let $\bar M_i\subset D^2$ be the disk
  bounded by $\partial_o M_i$.  We claim that there is a chain
  $\Delta_i$ in $\tau(\bar M_i)\cap L_0$ such that $\partial
  \Delta_i=\tau_\sharp[\partial_o M_i]$.
  
  We proceed by induction on the number $n$ of pieces of $D^2$
  contained in $\bar M_i$.  If $n=1$, then $\bar M_i=M_i$.  Then, as
  before, $\tau(M_i)$ is either a disk in $L_0$ or a disk in a copy
  of $L_A'$ which neighbors $L_0$.  In the first case, we can take
  $\Delta_i=\tau_\sharp[M_i]$.  In the second case, the lemma implies that
  $\tau_\sharp[\partial M_i]=0$, so we can take $\Delta_i=0$.

  Suppose that the claim is true for $n-1$ and let $M_i$ be a piece of
  $D^2$ such that $\tau(\partial_oM_i)\subset L_0$ and $\bar{M}_i$ is
  comprised of $n$ pieces of $D^2$.  If $\tau(M_i)$ is a punctured
  disk in $L_0$, then $\tau$ takes each inner boundary of $M_i$ to a curve in $L_0$
  and each inner boundary bounds a disk in $D^2$ with at most $n-1$ pieces.  By
  induction, each $\tau_\sharp[\partial_j M_i]$ bounds a chain $\Delta_{ij}$ in
  $\tau(\bar M_i)\cap L_0$.  Consequently, we can get the required filling of $\tau_\sharp[\partial_oM_i]$ as
  $$\Delta_i=\sum_j \Delta_{ij}+\tau_\sharp[M_i].$$

  On the other hand, if $\tau(M_i)$ is a punctured disk in a copy of
  $L_A'$, then Lemma~\ref{lem:LaPuncDiscs} implies that we can write
  $$\tau_\sharp[\partial_o M_i]=\sum_j a_{ij} \tau_\sharp[\partial_j M_i],$$
  where each $\partial_j M_i$ is an inner boundary component of
  $M_i$ which lies in $\partial L_0$.  By induction, each of these can
  be filled by a chain $\Delta_{ij}$ in $\tau(\bar M_i)\cap L_0$, and if
  $$\Delta_i=\sum_j a_{ij} \Delta_{ij},$$
  then $\Delta_i$ fills $\tau_\sharp[\partial_o M_i]$.
  
  Therefore, $\tau(\bar M_i)\cap L_0$ supports a chain filling $\gamma$. 
\end{proof}

\begin{proof}[Proof of Lemma~\ref{lem:LaPuncDiscs}]
  Let $Y$ be the complex used in the construction of $A$ and
  $\lozenge\subset Y$ be as in Sec.~\ref{sec:topology}.  Choose a
  basepoint $*\in Y$ such that $*\in \lozenge$.  Then by Theorem~8.6
  of \cite{BestBrad}, $L_A'$ is homotopy equivalent to an infinite
  wedge sum of copies of $Y$,
  $$Y_\infty=\bigvee_{\alpha \in S_A} Y_\alpha$$
  (in fact, $S_A$ corresponds to the set of vertices in $X_A\setminus
  L_A$).  Similarly, $\partial L_A'$ is made up of disjoint copies of
  $L_E\times\{1/2\}$; call these $L_{E,0},L_{E,1},\dots$, where
  $f(\partial_o M)\subset L_{E,0}$.  Each of these is homotopy
  equivalent to a wedge sum indexed by a subset $S_{E,i}\subset S_A$,
  $$\lozenge_{\infty,i}=\bigvee_{\alpha \in S_{E,i}} \lozenge_\alpha.$$
  The $S_{E,i}$'s partition $S_A$ into disjoint sets.  We can consider
  each of the $\lozenge_{\infty,i}$'s as a subset of $Y_\infty$, and
  we can define a homotopy equivalence $h:L_A'\to Y_\infty$
  that restricts to a homotopy equivalence $L_{E,i}\to
  \lozenge_{\infty,i}$ on each $L_{E,i}$.

  Let $*\in Y_\infty$ be the basepoint
  of the wedge sum.  Let 
  $$f':(M,\partial M)\to (Y_\infty, \bigcup_i \lozenge_{\infty,i})$$ 
  be a map which differs from $h\circ f$ by a small homotopy such that
  $f'^{-1}(*)$ is a graph in $M$.  We can further require that if
  $\lambda$ is a boundary component of $M$ and $f(\lambda)\subset
  L_{E,i}$, then $f'(\lambda)\subset \lozenge_{\infty,i}$.  Then
  $f'^{-1}(*)$ cuts $M$ into punctured disks $P_1,\dots, P_k$,
  and for each $i$, we can choose an $\alpha_i\in S_A$ such that
  $f'(P_i)\subset Y_{\alpha_i}$ and $f'(\partial P_i)\subset
  \lozenge_{\alpha_i}$.

  Consider $M$ as a subset of $\R^2$, embedded so that $\partial_o M$
  is the outer boundary of the subset.  For each $i$, let $D_i$ be the disk bounded by
  $\partial_iM$.  This embedding lets us choose an outer boundary
  component $\partial_oP_i$ for each $P_i$.  Furthermore, each closed
  curve $\lambda$ in $M$ has an inside, and we can use this to put a
  partial ordering on the boundary curves of the $P_i$'s.  If
  $\lambda$ and $\lambda'$ are two boundary curves, we write
  $\lambda\prec \lambda'$ if the inside of $\lambda$ is a subset of
  the inside of $\lambda'$.

  If $\lambda_1,\dots, \lambda_l$ are the inner boundary components of
  $P_i$, we claim that $f'_\sharp[\partial_o P_i]\in C_1(\lozenge_{\alpha_i})$ is a linear
  combination of the $f'_\sharp[\lambda_j]$'s.  All of these chains are in fact 1-cycles in
  $\lozenge_{\alpha_i}$, and since $Z_1(\lozenge_{\alpha_i})\cong \R$
  it's enough to show that if $f'_\sharp[\partial_o P_i]\ne 0$, then
  one of the $f'_\sharp[\lambda_j]$'s is nonzero too.  But if
  $f'_\sharp[\lambda_j]=0$ for all $j$, then $f'(\lambda_j)$ is a
  null-homotopic curve for all $j$, and so $f'(\partial_o P_i)$ is the
  boundary of a disk in $Y_{\alpha_i}$.  Since the inclusion
  $\lozenge\subset Y$ is $\pi_1$-injective, this means that if
  $f'_\sharp[\lambda_j]=0$ for all $j$, then
  $f'_\sharp[\partial_oM]=0$ as well, as desired.

  We claim that if $\lambda$ is a boundary curve of some $P_i$, then
  $f'_\sharp[\lambda]\in C_1(\bigcup_i \lozenge_{\infty,i})$ is a
  linear combination of the $f'_\sharp[\partial_j M]$'s.  We proceed
  by induction on the number of boundary curves $\lambda'$ with
  $\lambda'\prec \lambda$.  If this number is 0, then $\lambda$ is one
  of the $\partial_jM$, and there's nothing to prove.  
  Otherwise, the inside of $\lambda$ is a union of $P_i$'s and
  $D_i$'s, so $[\lambda]$ is a sum of $[\partial P_i]$'s and
  $[\partial_i M]$'s.  By induction, if $P_i$ is inside $\lambda$ and
  $\lambda\ne \partial_oP_i$, then $f'_\sharp[\partial P_i]$ is a
  linear combination of the $f'_\sharp[\partial_j M]$'s, so
  $f'_\sharp[\lambda]$ is a linear combination of the
  $f'_\sharp[\partial_j M]$'s too.

  Thus, there are $a_i\in \R$ such that 
  $$f'_\sharp[\partial_oM]=\sum_i a_i f'_\sharp[\partial_i M],$$
  where the equality is taken in $C_1(\bigcup_i \lozenge_{\infty,i})$, so 
  $$f_\sharp[\partial_oM]=\sum_i a_i f_\sharp[\partial_i M]$$
  in $H_1(\partial L_A')$.  Since $\partial L_A'$ is 1-dimensional,
  the equality in fact holds in $C_1(\partial L_A')$ as well.
  Finally, $C_1(\partial L_A')=\bigoplus_i C_1(L_{E,i})$, and if we
  project the above equation to the $C_1(L_{E,0})$ factor, we get
  $$f_\sharp[\partial_oM]=\sum_{\{i\mid f(\partial_i M)\subset L_{E,0}\}}  a_i f_\sharp[\partial_i M],$$
  as desired.
\end{proof}

\bibliographystyle{siam}
\bibliography{modRAAG.bib}

\begin{thebibliography}{10}

\bibitem{ABDDY}
{\sc A.~Abrams, N.~Brady, P.~Dani, M.~Duchin, and R.~Young}, {\em Pushing
  fillings in right-angled {A}rtin groups}, arXiv:1004.4253.

\bibitem{BaumslagBridsonMillerShort}
{\sc G.~Baumslag, M.~R. Bridson, C.~F. Miller, III, and H.~Short}, {\em
  Finitely presented subgroups of automatic groups and their isoperimetric
  functions}, J. London Math. Soc. (2), 56 (1997), pp.~292--304.

\bibitem{BestBrad}
{\sc M.~Bestvina and N.~Brady}, {\em Morse theory and finiteness properties of
  groups}, Invent. Math., 129 (1997), pp.~445--470.

\bibitem{BraBriForSha}
{\sc N.~Brady, M.~R. Bridson, M.~Forester, and K.~Shankar}, {\em Snowflake
  groups, {P}erron-{F}robenius eigenvalues and isoperimetric spectra}, Geom.
  Topol., 13 (2009), pp.~141--187.

\bibitem{BrGuLe}
{\sc N.~Brady, D.~Guralnik, and S.~R. Lee}, {\em Dehn functions and finiteness
  properties of subgroups of perturbed right-angled artin groups},
  arXiv:1102.5551.

\bibitem{BradyNonPos}
{\sc N.~Brady, T.~Riley, and H.~Short}, {\em The geometry of the word problem
  for finitely generated groups}, Advanced Courses in Mathematics. CRM
  Barcelona, Birkh\"auser Verlag, Basel, 2007.
\newblock Papers from the Advanced Course held in Barcelona, July 5--15, 2005.

\bibitem{BridsonSurv}
{\sc M.~R. Bridson}, {\em The geometry of the word problem}, in Invitations to
  geometry and topology, vol.~7 of Oxf. Grad. Texts Math., Oxford Univ. Press,
  Oxford, 2002, pp.~29--91.

\bibitem{BridHaef}
{\sc M.~R. Bridson and A.~Haefliger}, {\em Metric spaces of non-positive
  curvature}, vol.~319 of Grundlehren der Mathematischen Wissenschaften
  [Fundamental Principles of Mathematical Sciences], Springer-Verlag, Berlin,
  1999.

\bibitem{Charney}
{\sc R.~Charney}, {\em An introduction to right-angled {A}rtin groups}, Geom.
  Dedicata, 125 (2007), pp.~141--158.

\bibitem{GroFRM}
{\sc M.~Gromov}, {\em Filling {R}iemannian manifolds}, J. Differential Geom.,
  18 (1983), pp.~1--147.

\bibitem{GroAsymp}
{\sc M.~Gromov}, {\em Asymptotic invariants of infinite groups}, in Geometric
  group theory, {V}ol.\ 2 ({S}ussex, 1991), vol.~182 of London Math. Soc.
  Lecture Note Ser., Cambridge Univ. Press, Cambridge, 1993, pp.~1--295.

\bibitem{Wenger08}
{\sc S.~Wenger}, {\em A short proof of {G}romov's filling inequality}, Proc.
  Amer. Math. Soc., 136 (2008), pp.~2937--2941.

\bibitem{White}
{\sc B.~White}, {\em Mappings that minimize area in their homotopy classes}, J.
  Differential Geom., 20 (1984), pp.~433--446.

\bibitem{YoungHom}
{\sc R.~Young}, {\em Homological and homotopical higher-order filling
  functions}, Groups Geom. Dyn., 5 (2011), pp.~683--690.

\end{thebibliography}

\earticle
\end{document}